\documentclass[preprint,12pt]{elsarticle}
\usepackage{amsthm,amsmath,amssymb}
\usepackage[colorlinks=true,citecolor=black,linkcolor=black,urlcolor=blue]{hyperref}
\usepackage{graphicx}
\usepackage{float}
\usepackage{mathtools}
\usepackage{epstopdf}
\usepackage{epsfig}
\usepackage{subfigure}
\usepackage{multirow}

\theoremstyle{plain}
\newtheorem{theorem}{Theorem}
\newtheorem{lemma}[theorem]{Lemma}

\newtheorem{proposition}[theorem]{Proposition}

\theoremstyle{definition}
\newtheorem{definition}[theorem]{Definition}

\newtheorem{question}[theorem]{Question}

\theoremstyle{remark}
\newtheorem{remark}[theorem]{Remark}

\title{Partial-twuality polynomials of delta-matroids}
\author{Qi Yan\\
\small School of Mathematics\\[-0.8ex]
\small China University of Mining and Technology\\[-0.8ex]
\small P. R. China\\
Xian'an Jin\footnote{Corresponding author.}\\
\small School of Mathematical Sciences\\[-0.8ex]
\small Xiamen University\\[-0.8ex]
\small P. R. China\\
\small{\tt Email:qiyan@cumt.edu.cn; xajin@xmu.edu.cn}
}
\date{}

\begin{document}
\begin{abstract}
Gross, Mansour and Tucker introduced the partial-twuality polynomial of a ribbon graph.
Chumutov and Vignes-Tourneret posed a problem: it would be interesting to know whether the partial duality polynomial and the related conjectures would make sense for general delta-matroids.
In this paper we consider analogues of partial-twuality polynomials for delta-matroids.
Various possible properties of partial-twuality polynomials of set systems are studied.
We discuss the numerical implications of partial-twualities on a single element and prove that the intersection graphs can determine the partial-twuality polynomials of bouquets and normal binary delta-matroids, respectively. Finally, we give a characterization of vf-safe delta-matroids whose partial-twuality  polynomials have only one term.
\end{abstract}
\maketitle

\section{Introduction}
In \cite{WIL}, Wilson found that the two long-standing duality operators $\delta$
(geometric duality) and $\tau$ (Petrie duality) generate a group of six ribbon graph operators, that is,
every other composition of $\delta$ and $\tau$ is equivalent to one of the five operators $\delta$, $\tau$, $\delta\tau$, $\tau\delta$, $\delta\tau\delta$, or to the identity operator. Abrams and Ellis-Monaghan \cite{ABEL} called the five operators twualities. The partial (geometric) dual with respect to a subset of edges of a ribbon graph was introduced by Chmutov \cite{CG} in order to unify various connections between the Jones-Kauffman and Bollob\'{a}s-Riordan polynomials. Ellis-Monaghan and Moffatt \cite{EM} generalized this partial-duality construction to the other four operators, which they called partial-twualities.

Gross, Mansour and Tucker \cite{GMT, GMT2} introduced the partial-twuality polynomial for $\delta, \tau, \delta\tau, \tau\delta$, and $\delta\tau\delta$. Various basic properties of partial-twuality polynomials were studied, including
interpolation and log-concavity.  Recently, Chumutov and Vignes-Tourneret \cite{SCFV} posed the following question:

\begin{question}\cite{SCFV}
Ribbon graphs may be considered from the point of view of delta-matroid. In this way the concepts of partial (geometric) duality and genus can be interpreted in terms of delta-matroids \cite{CISR, CMNR}. It would be interesting to know whether the partial-$\delta$ polynomial and the related conjectures would make sense for general delta-matroids.
\end{question}

In \cite{QYXJ3}, we showed that the partial-$\delta$ polynomials have delta-matroid analogues. We introduced the twist polynomials of delta-matroids and discussed their basic properties for delta-matroids. Chun et al. \cite{CISR} showed that the loop complemenation is the delta-matroid analogue of partial Petriality. In this paper we consider analogues of other partial-twuality polynomials for delta-matroids.

This paper is organised as follows. In Section 2 we recall the definition of partial-twuality polynomials of ribbon graphs. Analogously, we introduce the partial-twuality polynomials of set systems. In Section 3, various possible properties of partial-twuality polynomials of set systems are studied. In Section 4 we discuss the numerical implications of partial-twualities on a single element and the interpolation. In Section 5, we prove that the intersection graphs can determine the partial-twuality polynomials of bouquets and normal binary delta-matroids, respectively. Here we provide an answer to the question \cite{QYXJ2}: can one derive something from bouquets that could determine the partial-twuality polynomial completely. In Section 6 we give a characterization of vf-safe delta-matroids whose partial-twuality  polynomials have only one term.

\section{Preliminaries}

\subsection{Set systems and widths}

A \emph{set system} is a pair $D=(E, \mathcal{F})$ of a finite set $E$ together with a collection $\mathcal{F}$ of subsets of $E$. The set $E$ is called the \emph{ground set} and the elements of $\mathcal{F}$ are the \emph{feasible sets}. We often use $\mathcal{F}(D)$ to denote the set of feasible sets of $D$.
$D$ is \emph{proper} if $\mathcal{F}\neq \emptyset$, and is \emph{normal} (respectively, \emph{dual normal}) if the empty set (respectively, the ground set) is feasible.
The \emph{direct sum} of two set systems $D=(E, \mathcal{F})$ and $\widetilde{D}=(\widetilde{E}, \widetilde{\mathcal{F}})$ with disjoint ground sets $E$ and $\widetilde{E}$, written $D\oplus \widetilde{D}$, is defined to be $$D\oplus \widetilde{D}:=(E\cup \widetilde{E}, \{F\cup \widetilde{F}: F\in \mathcal{F}~\text{and}~\widetilde{F}\in \widetilde{\mathcal{F}}\}).$$

As introduced by Bouchet in \cite{AB1}, a \emph{delta-matroid} is a proper set system $D=(E, \mathcal{F})$ such that if $X, Y \in \mathcal{F}$ and $u\in X\Delta Y$, then there is $v\in X\Delta Y$ (possibly  $v=u$ ) such that $X\Delta \{u, v\}\in \mathcal{F}$.  Here $$X\Delta Y:=(X\cup Y)-(X\cap Y)$$  is the usual symmetric difference of sets. Note that the maximum gap in the collection of sizes of feasible sets of a delta-matroid is two \cite{Moff}.


For a set system $D=(E, \mathcal{F})$, let $\mathcal{F}_{max}(D)$ and $\mathcal{F}_{min}(D)$ be the collections of maximum and minimum cardinality feasible sets of $D$, respectively. Let $D_{max}:=(E, \mathcal{F}_{max}(D))$ and $D_{min}:=(E, \mathcal{F}_{min}(D))$. Let $r(D_{max})$  and $r(D_{min})$ denote the sizes of largest and  smallest feasible sets of $D$, respectively.
The  \emph{width} of $D$, denote by $w(D)$, is defined by $$w(D):=r(D_{max})-r(D_{min}).$$
For all non-negative integers $i\leq w(D)$, let $$\mathcal{F}_{max-i}(D)=\{F\in\mathcal{F}: |F|=r(D_{max})-i\}$$ and $$\mathcal{F}_{min+i}(D)=\{F\in\mathcal{F}: |F|=r(D_{min})+i\}.$$



\subsection{Partial-twualities of set systems}

We will consider the operations of twisting and loop complementation on set systems.
Twisting was introduced by Bouchet in \cite{AB1}, and loop complementation by Brijder and
Hoogeboom in \cite{BRHH}.

Let $D=(E, \mathcal{F})$ be a set system. For $A\subseteq E$, the \emph{twist} of $D$ with respect to $A$, denoted by $D^{*|A}$, is given by $$(E, \{A\Delta X: X\in \mathcal{F}\}).$$ The \emph{$\ast$-dual} of $D$, written $D^{*}$, is equal to $D^{*|E}$. Note that $\ast$-duality preserves width. Throughout the paper, we will often omit the set brackets in the case of a single element set. For example, we write $D^{*|e}$ instead of $D^{*|\{e\}}$.

Let $D=(E, \mathcal{F})$ be a set system and $e\in E$. Then $D^{\times|e}$ is defined to be the set system $(E, \mathcal{F}')$, where $$\mathcal{F}'=\mathcal{F}\Delta \{F\cup e: F\in \mathcal{F}~\text{and}~e\notin F\}.$$
If $e_{1}, e_{2}\in E$ then $$(D^{\times|e_{1}})^{\times|e_{2}}=(D^{\times|e_{2}})^{\times|e_{1}}.$$ This means that if $A=\{e_{1}, \cdots, e_{m}\}\subseteq E$ we can unambiguously define the \emph{loop complementation} \cite{BRHH} of $D$ on $A$, by $$D^{\times|A}:=(\cdots(D^{\times|e_{1}})^{\times|e_{2}}\cdots)^{\times|e_{m}}.$$

It is straightforward to show that the twist of a delta-matroid is a delta-matroid \cite{AB1},
but the set of delta-matroids is not closed under loop complementation (see, for example, \cite{CISR}).  Thus, we often restrict our attention to a class of delta-matroids that is closed under loop complementation. A delta-matroid $D=(E, \mathcal{F})$ is said to be \emph{vf-safe} \cite{CISR} if the application of every sequence of twists and loop complementations results in a delta-matroid.

In \cite{BRHH} it was shown that twists and loop complementations give rise to an action of the symmetric group $S_{3}$, with the presentation $$S_{3}\cong \mathcal{B}:=<\ast, \times~|~\ast^{2}, \times^{2}, (\ast\times)^{3}>,$$ on set systems. If $D=(E, \mathcal{F})$ is a set system, $e\in E$ and $a=a_{1}a_{2}\cdots a_{n}$ is a word in the alphabet $\{\ast, \times\}$, then
$$D^{a|e}:=(\cdots(D^{a_{1}|e})^{a_{2}|e}\cdots)^{a_{n}|e}.$$
Note that the operators $\ast$ and $\times$ on different elements commute \cite{BRHH}.
If $A=\{e_{1}, \cdots, e_{m}\}\subseteq E$, we can unambiguously define
$$D^{a|A}:=(\cdots(D^{a|e_{1}})^{a|e_{2}}\cdots)^{a|e_{m}}.$$
Let $D_{1}=(E, \mathcal{F})$ and $D_{2}$ be set systems. For $\bullet\in \{\ast, \times, \ast\times, \times\ast, \ast\times\ast\}$, we say that $D_{2}$ is a \emph{partial-$\bullet$ dual} of $D_{1}$ if there exists $A\subseteq E$ such that $D_{2}={D_{1}}^{\bullet|A}$.


\subsection{Partial-twualities of ribbon graphs}
Ribbon graphs are well-known to be equivalent to cellularly embedded graphs. The reader is referred to \cite{EM1,EM} for further details about ribbon graphs. A \emph{quasi-tree} is a ribbon graph with one boundary component.
Let $G=(V, E)$ be a ribbon graph and let $$\mathcal{F}:=\{F\subseteq E(G): \text{$F$ is the edge set of a spanning quasi-tree of $G$}\}.$$ We call $D(G)=:(E, \mathcal{F})$ the \emph{delta-matroid} \cite{CMNR} of $G$.
We say a delta-matroid is \emph{ribbon-graphic} if it is equal to the delta-matroid of some ribbon graph. Note that ribbon-graphic delta-matroids are vf-safe \cite{CISR}.

For a ribbon graph $G$ and a subset $A$ of its edge-ribbons $E(G)$,  the \emph{partial dual} $G^{\delta|A}$  \cite{CG} of $G$ with respect to $A$ is a ribbon graph obtained from $G$ by gluing a disc to $G$ along each boundary component of the spanning ribbon subgraph $(V (G), A)$ (such discs will be the vertex-discs of $G^{\delta|A}$), removing the interiors of all vertex-discs of $G$ and keeping the edge-ribbons unchanged.

Let $G$ be a ribbon graph and $A\subseteq E(G)$. Then the \emph{partial Petrial} $G^{\tau|A}$ \cite{EM1} of $G$ with respect to $A$ is the ribbon graph obtained from $G$ by adding a half-twist to each of the edges in $A$.

In \cite{EM1} it was shown that the partial dual, $\delta$, and the partial Petrial, $\tau$, give rise to an action of the symmetric group $S_{3}$, with the presentation $$S_{3}\cong \mathcal{R}:=<\delta, \tau~|~\delta^{2}, \tau^{2}, (\delta\tau)^{3}>,$$ on ribbon graphs. If $G$ is a ribbon graph, $e\in E(G)$ and $a=a_{1}a_{2}\cdots a_{n}$ is a word in the alphabet $\{\delta, \tau\}$, then
$$G^{a|e}:=(\cdots(G^{a_{1}|e})^{a_{2}|e}\cdots)^{a_{n}|e}.$$
Observe that the partial dual and the partial Petrial commute when applied to different edges \cite{EM1}. If $A=\{e_{1}, \cdots, e_{m}\}\subseteq E$, we define
$$G^{a|A}:=(\cdots(G^{a|e_{1}})^{a|e_{2}}\cdots)^{a|e_{m}}.$$

Let $G_{1}$ and $G_{2}$ be ribbon graphs. For $\bullet \in \{\delta, \tau, \delta\tau, \tau\delta, \delta\tau\delta\}$, we say that $G_{2}$ is a partial-$\bullet$ dual \cite{EM1} of $G_{1}$ if there exists $A\subseteq E(G_{1})$ such that $G_{2}={G_{1}}^{\bullet|A}$.


\subsection{Partial-twuality polynomials of ribbon graphs and set systems}

Gross, Mansour and Tucker \cite{GMT2} introduced the concept of partial-twuality polynomials of ribbon graphs  as follows.

\begin{definition}[\cite{GMT2}]
For $\bullet \in \mathcal{R}$, we define the partial-$\bullet$ polynomial for any ribbon graph $G$ to be the generating function
$$^{\partial}\varepsilon_{G}^{\bullet}(z):=\sum_{A\subseteq E(G)}z^{\varepsilon(G^{\bullet\mid A})}$$
that enumerates all partial-$\bullet$ duals of $G$ by Euler genus.
\end{definition}

Analogously, we define the partial-twuality polynomials of set systems as follows.


\begin{definition}
For $\bullet \in \mathcal{B}$, the \emph{partial-$\bullet$ polynomial} of any set system $D=(E, \mathcal{F})$ is defined to be the generating function
$$^{\partial}w_{D}^{\bullet}(z):=\sum_{A\subseteq E}z^{w(D^{\bullet\mid A})}$$
that enumerates all partial-$\bullet$ duals of $D$ by width.
\end{definition}

\subsection{Binary and intersection graphs}

For a finite set $E$, let $C$ be a symmetric $|E|$ by $|E|$ matrix over $GF(2)$, with rows and columns indexed, in the same order, by the elements of $E$. Let $C[A]$ be the principal submatrix of $C$ induced by the set $A\subseteq E$. We define the set system $D(C)=(E, \mathcal{F})$ with $$\mathcal{F}:=\{A\subseteq E: C[A] \mbox{ is non-singular}\}.$$  By convention $C[\emptyset]$ is non-singular. Then $D(C)$ is a delta-matroid \cite{AB4}. A delta-matroid is said to be \emph{binary} if it has a twist that is isomorphic to $D(C)$ for some symmetric matrix $C$ over $GF(2)$.

Let $D=(E, \mathcal{F})$ be a normal binary delta-matroid. Then there exists a  unique symmetric $|E|$ by $|E|$ matrix $C$ over $GF(2)$ such that $D=D(C)$ \cite{Moff, QYXJ3}. The \emph{intersection graph} $G_{D}$ of $D$ is the graph with vertex set $E$ and in which two vertices $u$ and $v$ of $G_{D}$ are adjacent if and only if $C_{u, v}=1$.   
A \emph{bouquet} is a ribbon graph with only one vertex.  If $B$ is a bouquet, then $D(B)$ is a normal binary delta-matroid \cite{CMNR}. The \emph{intersection graph} $I(B)$ of a bouquet $B$ is the graph $G_{D(B)}$.

Conversely, recall that a \emph{looped simple graph} \cite{Moff} is a graph obtained from a simple graph by adding (exactly) one loop to some of its vertices. The adjacency matrix $A(G)$ of a looped simple graph $G$ is the matrix over $GF(2)$ whose rows and columns correspond to the vertices of $G$; and where, $A(G)_{u, v}=1$ if and only if $u$ and $v$ are adjacent in $G$ and $A(G)_{u, u}=1$ if and only if there is a loop at $u$.
Let $D$ be a normal binary delta-matroid. It obvious that $D=D(A(G_{D}))$.

\subsection{Primal and dual types}

Let $D=(E, \mathcal{F})$ be a proper set system. An element $e\in E$ contained in no feasible set of $D$ is said to be a \emph{loop}.

\begin{definition}[\cite{CMNR}]
Let $D=(E, \mathcal{F})$ be a set system and $e\in E$. Then
\begin{description}
  \item[(1)] $e$ is a \emph{ribbon loop} if $e$ is a loop in $D_{min}$;
  \item[(2)] A ribbon loop $e$ is \emph{non-orientable} if $e$ is a ribbon loop in $D^{*|e}$ and is \emph{orientable} otherwise.
\end{description}
\end{definition}
Let $D=(E, \mathcal{F})$ be a set system and $e\in E$. The \emph{primal type} of $e$ is $p, u,$ or $t$ in $D$, if $e$ is a non-ribbon loop, an orientable loop, or a non-orientable loop,  respectively, in $D$. The primal type of $e$ in $D^{\ast}$ is called the \emph{dual type} of $e$ in $D$. In combination, the primal and dual types of $e$ in $D$ are called the \emph{type} of $e$ in $D$, which is denoted by a juxtaposed pair of letters representing the primal and dual types of $e$ in $D$.  For example, the type $pu$ means that the primal and dual types of $e$ are $p$ and $u$, respectively, in $D$. We observe that

\begin{description}
  \item[(1)] The primal type of $e$ is $p$ in $D$ if and only if there exists $A\in \mathcal{F}_{min}(D)$ such that $e\in A;$
  \item[(2)] The dual type of $e$ is $p$ in $D$ if and only if there exists $A\in \mathcal{F}_{max}(D)$ such that $e\notin A;$
  \item[(3)] The primal type of $e$ is $u$ in $D$ if and only if for any $A\in \mathcal{F}_{min}(D)\cup \mathcal{F}_{min+1}(D)$, $e\notin A;$
  \item[(4)] The dual type of $e$ is $u$ in $D$ if and only if for any $A\in \mathcal{F}_{max}(D)\cup \mathcal{F}_{max-1}(D)$, $e\in A;$
  \item[(5)] The primal type of $e$ is $t$ in $D$ if and only if for any $A\in \mathcal{F}_{min}(D)$, $e\notin A$, and there exists $B\in \mathcal{F}_{min+1}(D)$ such that $e\in B;$
  \item[(6)] The dual type of $e$ is $t$ in $D$ if and only if for any $A\in \mathcal{F}_{max}(D)$, $e\in A$, and there exists $B\in \mathcal{F}_{max-1}(D)$ such that $e\notin B.$
\end{description}

\section{ Some properties of partial-twuality polynomials}
Various possible properties of partial-twuality polynomials of ribbon graphs were studied by Gross, Mansour and Tucker in \cite{GMT, GMT2}. In this section we discuss the analogous results on set systems or delta-matroids.









\begin{proposition}
Let $D=(E, \mathcal{F})$ and $\widetilde{D}=(\widetilde{E}, \widetilde{\mathcal{F}})$ be set systems. Then for any $\bullet\in \mathcal{B}$,
\begin{description}
  \item[(1)] $^{\partial}w_{D}^{\bullet}(1)=2^{\mid E\mid};$
  \item[(2)] $^{\partial}w_{D}^{\bullet}(z)$ has degree at most $|E|;$
  \item[(3)]$^{\partial}w_{D\oplus \widetilde{D}}^{\bullet}(z)={^{\partial}w_{D}^{\bullet}(z)} {^{\partial}w_{\widetilde{D}}^{\bullet}(z)}.$
\end{description}
\end{proposition}

\begin{proof}
For (1), the value $^{\partial}w_{D}^{\bullet}(1)$ counts the total number of partial-$\bullet$ duals of $D$, which is $2^{|E|}$. For any subset $A\subseteq E$, if $B\in\mathcal{F}(D^{\bullet|A})$, then $\emptyset\subseteq B\subseteq E$. We have $r({D^{\bullet|A}}_{min})\geq0$ and $r({D^{\bullet|A}}_{max})\leq |E|$. Thus $0\leq w(D^{\bullet|A})\leq |E|$ and (2)  then follows. For any subset $C\subseteq E\cup \widetilde{E}$, we have
$$(D\oplus \widetilde{D})^{\bullet|C}=D^{\bullet|(C\cap E)}\oplus \widetilde{D}^{\bullet|(C\cap \widetilde{E})}.$$
Then $$^{\partial}w_{D\oplus \widetilde{D}}^{\bullet}(z)={^{\partial}w_{D}^{\bullet}(z)} {^{\partial}w_{\widetilde{D}}^{\bullet}(z)},$$
by the additivity of  width over the direct sum, from which (3) follows.
\end{proof}

\begin{proposition}\label{pro 1}
Let $D=(E, \mathcal{F})$ be a set system and $A\subseteq E$. Then $$^{\partial}w_{D}^{\bullet}(z)={^{\partial}w^{\bullet}_{D^{\bullet|A}}(z)}$$ for $\bullet\in \{\ast, \times, \ast\times\ast\}.$
\end{proposition}

\begin{proof}
This is because the set of all loop complementations of $D$ is the same as that of $D^{\times|A}$. The same reasoning applies to the operators $\ast$ and $\ast\times\ast$.
\end{proof}

\begin{remark}
Proposition \ref{pro 1} is not true for the operators $\ast\times$ and $\times\ast$. For example, let $D=(E, \mathcal{F})$ with $E=\{1\}$ and $\mathcal{F}=\{\emptyset, \{1\}\}$. Then $D^{\ast\times|1}=(\{1\}, \{\emptyset\})$ and $D^{\times\ast|1}=(\{1\}, \{\{1\}\})$. We have $$^{\partial}w_{D}^{\ast\times}(z)={^{\partial}w_{D}^{\times\ast}(z)}=1+z$$ and $$^{\partial}w^{\ast\times}_{D^{\ast\times|1}}(z)={^{\partial}w^{\times\ast}_{D^{\times\ast|1}}(z)}=2.$$ Obviously, $^{\partial}w_{D}^{\ast\times}(z)\neq{^{\partial}w^{\ast\times}_{D^{\ast\times|1}}(z)}$ and
$^{\partial}w_{D}^{\times\ast}(z)\neq{^{\partial}w^{\times\ast}_{D^{\times\ast|1}}(z)}.$
\end{remark}




\begin{lemma}[\cite{BRHH}\label{lem 2}]
Let $D=(E, \mathcal{F})$ be a set system and $A\subseteq E$. Then $$\mathcal{F}_{min}(D)=\mathcal{F}_{min}(D^{\times|A})$$ and $$\mathcal{F}_{max}(D)=\mathcal{F}_{max}(D^{\ast\times\ast|A})=\mathcal{F}_{max}(D^{\times\ast\times|A}).$$
\end{lemma}

\begin{proposition}\label{the 1}
Let $D=(E, \mathcal{F})$ be a set system and $A\in \mathcal{F}_{min}(D), B\in \mathcal{F}_{min}(D^{\ast})$. Then
\begin{description}
  \item[(1)] $D^{\bullet|A}$ is normal for $\bullet\in \{\ast, \ast\times, \times\ast, \ast\times\ast\};$
  \item[(2)] $D^{\times|B}$ is dual normal.
\end{description}
\end{proposition}

\begin{proof}
{\bf (1)} We may assume that $A\neq\emptyset$, otherwise the conclusion is trivial. For any $e\in A$, since $A\in \mathcal{F}_{min}(D)$, it follows that $A\in \mathcal{F}_{min}(D^{\times|e})$ by Lemma \ref{lem 2} and $A-e\in \mathcal{F}_{min}(D^{*|e})$. Then $A-e\in \mathcal{F}_{min}(D^{*\times|e})$ by Lemma \ref{lem 2} and $A-e\in \mathcal{F}_{min}(D^{\times \ast|e})$. Thus $A-e\in \mathcal{F}_{min}(D^{\times\ast\times|e})$ by Lemma \ref{lem 2}. From the above, we have $A-e\in \mathcal{F}_{min}(D^{\bullet|e})$ for $\bullet\in \{\ast, \ast\times, \times\ast, \ast\times\ast\}$. In the same manner we can see that $\emptyset \in \mathcal{F}_{min}(D^{\bullet|A})$ for $\bullet\in \{\ast, \ast\times, \times\ast, \ast\times\ast\}$ and conclusion (1) then follows.

{\bf (2)} Since $B\in \mathcal{F}_{min}(D^{\ast})$, it follows that $E-B\in \mathcal{F}_{max}(D)$. Then $E-B\in \mathcal{F}_{max}(D^{\times\ast\times|B})$ by Lemma \ref{lem 2}, that is, $E-B\in \mathcal{F}(D^{\times\ast\times|B})$. Thus $E\in \mathcal{F}(D^{\times\ast\times\ast|B})$. Obviously,
$$E\in \mathcal{F}_{max}(D^{\times\ast\times\ast|B})=\mathcal{F}_{max}((D^{\times|B})^{\ast\times\ast|B}).$$
Then $E\in \mathcal{F}_{max}(D^{\times|B})$ by Lemma \ref{lem 2}, that is, $E\in \mathcal{F}(D^{\times|B})$. Thus $D^{\times|B}$ is dual normal.
\end{proof}

\begin{remark}

For investigation of partial-$\bullet$ polynomials of set systems for $\bullet\in \{\ast, \ast\times\ast\}$, Propositions \ref{pro 1} and \ref{the 1} motivate us to focus on normal set systems, and for $\bullet=\times$, to focus on dual normal set systems. But for $\ast\times$ or $\times\ast$, we cannot just focus on normal set systems. For example, let $D=(\{1\}, \{\{1\}\})$. We have $^{\partial}w_{D}^{\times\ast}(z)=2$. Observe that all normal set systems with ground set $\{1\}$ are $D_{1}=(\{1\}, \{\emptyset\})$ and $D_{2}=(\{1\}, \{\emptyset, \{1\}\}).$
Since $^{\partial}w_{D_{1}}^{\times\ast}(z)={^{\partial}w_{D_{2}}^{\times\ast}(z)}=1+z,$ it follows that there is no normal set system $D'$ such that $^{\partial}w_{D'}^{\times\ast}(z)={^{\partial}w_{D}^{\times\ast}(z)}$.
\end{remark}

The following theorem provides a link between partial-$\ast\bullet\ast$ and partial-$\bullet$ polynomials of set systems.

\begin{theorem}\label{th 2}
Let $D=(E, \mathcal{F})$ be a set system. Then for any $\bullet\in \mathcal{B}$,
$$^{\partial}w_{D}^{\ast\bullet\ast}(z)={^{\partial}w_{D^{\ast}}^{\bullet}(z)}.$$
\end{theorem}

\begin{proof}
For any $A\subseteq E$, we observe that doing partial-$\ast\bullet\ast$ on $A$ is the same as first doing $\ast$ to $E$, then doing $\bullet$ to $A$, and then doing $\ast$ to $E$ again, that is, $$D^{\ast\bullet\ast|A}=((D^{\ast})^{\bullet|A})^{\ast}.$$
Since $\ast$-duality preserves width, it follows that $$w(D^{\ast\bullet\ast|A})=w(((D^{\ast})^{\bullet|A})^{\ast})=w((D^{\ast})^{\bullet|A}).$$
Thus the partial-$\ast\bullet\ast$ polynomial of $D$ is identical to the partial-$\bullet$ polynomial of $D^{\ast}$.
\end{proof}

\section{Partial-twuality for a single element}

In this section, we discuss the numerical implications of partial-twualities on a single element $e$,
depending on the type of $e$.

\begin{lemma}[\cite{BRHH2}\label{lem 4}]
Let $D=(E, \mathcal{F})$ be a delta-matroid and $e\in E$ such that $r(D_{min})=r({D^{\ast|e}}_{min})$. Then $\mathcal{F}_{min}(D)=\mathcal{F}_{min}({D^{\ast|e}})$.
\end{lemma}

\begin{remark}
Lemma \ref{lem 4} is not true for set systems. For example, let $$D=(\{1, 2, 3\}, \{\{1\}, \{2, 3\}\}).$$ We know $r(D_{min})=r({D^{\ast|2}}_{min})=1$. But $$\mathcal{F}_{min}(D)=\{\{1\}\}$$
and
$$\mathcal{F}_{min}({D^{\ast|2}})=\{\{3\}\}.$$
\end{remark}

\begin{lemma}\label{lem 7}
Let $D=(E, \mathcal{F})$ be a delta-matroid and $e\in E$. If $e$ is a non-orientable loop, then for any $A\in \mathcal{F}_{min}(D)$, $A\cup e\in \mathcal{F}(D)$.
\end{lemma}

\begin{proof}
Since the primal type of $e$ is $t$ in $D$, it follows that $e\notin A$ and there exists $B\in \mathcal{F}_{min+1}(D)$ such that $e\in B$. Then $B-e\in \mathcal{F}_{min}(D^{\ast|e})$. We have $r(D_{min})=r({D^{\ast|e}}_{min})$ and hence $\mathcal{F}_{min}(D)=\mathcal{F}_{min}({D^{\ast|e}})$ by Lemma \ref{lem 4}. Then $A\in \mathcal{F}_{min}(D^{\ast|e})$, that is, $A\in \mathcal{F}(D^{\ast|e})$. Thus, $A\cup e\in \mathcal{F}(D)$.
\end{proof}

\begin{lemma}[\cite{BECN}]\label{lem 9}
If $X$ is any feasible set in a delta-matroid $D$, then there exist $A\in \mathcal{F}_{min}(D)$ and $B\in \mathcal{F}_{max}(D)$ such that $A\subseteq X \subseteq B$.
\end{lemma}

\begin{theorem}\label{lem 5}
Let $D=(E, \mathcal{F})$ be a vf-safe delta-matroid and $e\in E$. Table \ref{ta 1} gives the value of $w(D^{\bullet|e})-w(D)$ for any $\bullet\in \mathcal{B}$.
\begin{table}[!t]
\centering
\caption{The difference $w(D^{\bullet|e})-w(D)$ for any $\bullet\in \mathcal{B}$.}
\label{ta 1}
\begin{tabular}{c|ccccc}
  \hline
  Type of $e$& $~~~\ast~~~$ & $~~~\times~~~$ & $~~~\ast\times~~~$ & $~~~\times\ast~~~$ & $~~~\ast\times\ast~~~$ \\\hline
  $pp$        & $+2$      & $+1$        & $+2$            & $+2$            & $+1$ \\
  $uu$        & $-2$      & 0         & $-1$            & $-1$            & 0 \\
  $pu$        & 0       & 0         & $+1$            & 0             & $+1$ \\
  $up$        & 0       & $+1$        & 0             & $+1$            & 0 \\
  $tp$        & $+1$      & $+1$        & $+1$            & 0             & $-1$ \\
  $tu$        & $-1$      & 0         & 0             & $-2$            & $-1$ \\
  $pt$        & $+1$      & $-1$        & 0             & $+1$            & $+1$ \\
  $ut$        & $-1$      & $-1$        & $-2$            & 0             & 0 \\
  $tt$        & 0       & $-1$        & $-1$            & $-1$            & $-1$ \\
  \hline
\end{tabular}
\end{table}
\end{theorem}

\begin{proof}
The three possible primal types (and dual types) of $e$ in $D$ are as follows:


\begin{description}
  \item[Case 1.] If the primal type of $e$ is $p$ in $D$, there exists $A\in \mathcal{F}_{min}(D)$ such that $e\in A$. Then $A-e\in \mathcal{F}_{min}(D^{*|e})$. Thus $$r({D^{*|e}}_{min})=r(D_{min})-1$$ and the primal types of $e$  are $u$ and $p$ in $D^{*|e}$ and $D^{\times|e}$, respectively.
  \item[Case 2.] If the primal type of $e$ is $u$ in $D$, then for any $A\in \mathcal{F}_{min}(D)\cup \mathcal{F}_{min+1}(D)$, $e\notin A.$ Thus $$r({D^{*|e}}_{min})=r(D_{min})+1$$ and the types of $e$  are $p$ and $t$ in $D^{*|e}$ and $D^{\times|e}$, respectively.
  \item[Case 3.] If the primal type of $e$ is $t$ in $D$, then for any $A\in \mathcal{F}_{min}(D)$, $e\notin A$, and there exists $B\in \mathcal{F}_{min+1}(D)$ such that $e\in B.$ Thus $$r({D^{*|e}}_{min})=r(D_{min})$$ and the primal types of $e$ is $t$ in $D^{*|e}$. By Lemma \ref{lem 7}, we have $A\cup e\in \mathcal{F}_{min+1}(D)$ for any $A\in \mathcal{F}_{min}(D)$. Then $A\cup e\notin \mathcal{F}(D^{\times|e}).$ Furthermore, we know that for any $B\in \mathcal{F}_{min+1}(D)$ containing $e$, $B-e\in \mathcal{F}_{min}(D)$,  otherwise there is no $A'\in \mathcal{F}_{min}(D)$ such that $A'\subseteq B$, contradicting Lemma \ref{lem 9}. Since $\mathcal{F}_{min}(D^{\times|e})=\mathcal{F}_{min}(D)$,  it follows that there is no $B'\in\mathcal{F}_{min}(D^{\times|e})\cup \mathcal{F}_{min+1}(D^{\times|e})$ such that $e\in B'$. Then the primal type of $e$ is $u$ in $D^{\times|e}$.
 \end{description}

       Here, we give a summary of Cases 1, 2 and 3 as shown in Table \ref{ta 2}.


\begin{table}[!t]
\centering
\caption{A summary of Cases 1, 2 and 3.}
\label{ta 2}
\begin{tabular}{llll}
  \hline
  \multicolumn{3}{c}{Primal type of $e$}&\multirow{2}{*}{$~~~~r({D^{*|e}}_{min})~~$}\\\cline{1-3}
  $D$~~~~&$~~~~D^{*|e}~~~~$&~~~~$D^{\times|e}~~~~$~~~~&\multirow{2}{*}{}\\\hline
  $p$    &   ~~~~~~$u$          &   ~~~~~~$p$   &   $~~~~r(D_{min})-1$            \\
  $u$    &   ~~~~~~$p$          &   ~~~~~~$t$   &   $~~~~r(D_{min})+1$           \\
  $t$    &   ~~~~~~$t$          &   ~~~~~~$u$   &   $~~~~r(D_{min})$            \\
  \hline
\end{tabular}
\end{table}

\begin{description}
  \item[Case 4.] If the  dual type of $e$ is $p$ in $D$, there exists $A\in \mathcal{F}_{max}(D)$ such that $e\notin A$. Then $A\cup e\in \mathcal{F}_{max}(D^{*|e})\cap \mathcal{F}_{max}(D^{\times|e})$. Thus $$r({D^{\ast|e}}_{max})=r({D^{\times|e}}_{max})=r(D_{max})+1$$ and the  dual types of $e$  are $u$ and $t$ in $D^{*|e}$ and $D^{\times|e}$, respectively.
  \item[Case 5.] If the dual type of $e$ is $u$ in $D$, then for any $A\in \mathcal{F}_{max}(D)\cup \mathcal{F}_{max-1}(D)$, $e\in A$. Thus $$r({D^{*|e}}_{max})=r(D_{max})-1$$ and $$r({D^{\times|e}}_{max})=r(D_{max})$$ and the  dual types of $e$  are $p$ and $u$ in $D^{*|e}$ and $D^{\times|e}$, respectively.
  \item[Case 6.] If the dual type of $e$ is $t$ in  $D$, then for any $A\in \mathcal{F}_{max}(D)$, $e\in A$. Thus $E-A\in \mathcal{F}_{min}(D^{*})$ and $(E-A)\cup e \in \mathcal{F}(D^{*})$ by Lemma \ref{lem 7}. It follows that $A-e\in \mathcal{F}_{max-1}(D)$. We have $$r({D^{*|e}}_{max})=r(D_{max})$$ and the dual type of $e$ is $t$ in $D^{*|e}$. Moreover, we observe that for any $B\in \mathcal{F}_{max-1}(D)$ not containing $e$, $B\cup e\in \mathcal{F}_{max}(D)$, otherwise there is no $B'\in \mathcal{F}_{max}(D)$ such that $B\subseteq B'$, contradicting Lemma \ref{lem 9}. It follows that $$r({D^{\times|e}}_{max})=r(D_{max})-1$$ and the dual type of $e$  is $p$ in $D^{\times|e}$, respectively.
\end{description}

Here, we provide a summary of Cases 4, 5 and 6 as shown in Table \ref{ta 3}.
\begin{table}[!t]
\centering
\caption{A summary of Cases 4, 5 and 6}
\label{ta 3}
\begin{tabular}{lllll}
  \hline
  \multicolumn{3}{c}{Dual type of $e$}&\multirow{2}{*}{~~$r({D^{*|e}}_{max})$~~}&\multirow{2}{*}{~~$r({D^{\times|e}}_{max})$~~}\\\cline{1-3}
  $D$&$~~D^{*|e}~~$&~~$D^{\times|e}~~$&\multirow{2}{*}{}&\multirow{2}{*}{}\\\hline
  $p$~~    &   ~~$u$          &~~$t$     &~~$r(D_{max})+1$ &  $~~r(D_{max})+1$          \\
  $u$~~    &   ~~$p$          &~~$u$     &~~$r(D_{max})-1$ &  $~~r(D_{max})$          \\
  $t$~~    &   ~~$t$          &~~$p$     &~~$r(D_{max})$   &  $~~r(D_{max})-1$        \\
  \hline
\end{tabular}
\end{table}
Then the the widths of $D^{*|e}$ and $D^{\times|e}$ can be calculated by Tables \ref{ta 2} and \ref{ta 3} as shown in Table \ref{ta 4}.
\begin{table}[!t]
\centering
\caption{The the widths of $D^{*|e}$ and $D^{\times|e}$ }
\label{ta 4}
\begin{tabular}{c|l|l|l|l|l}
  \hline
  Type of $e$&$r({D^{*|e}}_{min})$&$r({D^{*|e}}_{max})$&$r({D^{\times|e}}_{max})$&$w(D^{*|e})$&$w(D^{\times|e})$\\ \hline
  $pp$          &$r(D_{min})-1$       & $r(D_{max})+1$ &$r(D_{max})+1$      &$w(D)+2$ & $w(D)+1$           \\
  $uu$          &$r(D_{min})+1$       & $r(D_{max})-1$ &$r(D_{max})$     &$w(D)-2$  & $w(D)$         \\
  $pu$          &$r(D_{min})-1$       & $r(D_{max})-1$ &$r(D_{max})$    &$w(D)$   &  $w(D)$        \\
  $up$          &$r(D_{min})+1$       & $r(D_{max})+1$ &$r(D_{max})+1$      &$w(D)$ &  $w(D)+1$           \\
  $tp$          &$r(D_{min})$         & $r(D_{max})+1$ &$r(D_{max})+1$      &$w(D)+1$  & $w(D)+1$          \\
  $tu$          &$r(D_{min})$         & $r(D_{max})-1$ &$r(D_{max})$     &$w(D)-1$  &  $w(D)$          \\
  $pt$          &$r(D_{min})-1$       & $r(D_{max})$   &$r(D_{max})-1$      &$w(D)+1$ &  $w(D)-1$           \\
  $ut$          &$r(D_{min})+1$       & $r(D_{max})$   &$r(D_{max})-1$      &$w(D)-1$ &  $w(D)-1$         \\
  $tt$          &$r(D_{min})$         & $r(D_{max})$   &$r(D_{max})-1$      &$w(D)$  &  $w(D)-1$         \\
  \hline
\end{tabular}
\end{table}
Hence, the columns 2 and 3 of Table \ref{ta 1} are computed.
If the type of $e$ is $pp$ in $D$, then $$w(D^{*|e})=w(D)+2$$ and $$w(D^{\times|e})=w(D)+1,$$ and the types of $e$ are $uu$ and $pt$ in $D^{*|e}$ and $D^{\times|e}$, respectively.
Thus $$w(D^{\ast\times|e})=w((D^{*|e})^{\times|e})=w(D^{*|e})=w(D)+2,$$ and $$w(D^{\times\ast|e})=w((D^{\times|e})^{*|e})=w(D^{\times|e})+1=w(D)+2,$$
and the type of $e$ is $tu$ in $D^{\ast\times|e}$.
We have $$w(D^{\ast\times\ast|e})=w((D^{\ast\times|e})^{\ast|e})=w(D^{\ast\times|e})-1=w(D)+1.$$
The other entries in columns 4, 5 and 6 of Table \ref{ta 1} are computed similarly.
\end{proof}

The polynomial $p(z)=\sum\limits_{i=0}^{n}c_{i}z^{i}$ is said to have a \emph{gap of size $k$} \cite{GMT2} at coefficient $c_{i}$ if $c_{i-1}c_{i+k}\neq 0$ but $c_{i}=c_{i+1}=\cdots=c_{i+k-1}=0$. If the polynomial $p(z)$ is nonzero
and has no gaps, we call it \emph{interpolating}.
\begin{proposition}
For any vf-safe delta-matroid $D$, the following statements hold:
\begin{description}
  \item[(1)] $^{\partial}w_{D}^{\bullet}(z)$ is interpolating for $\bullet=\times$ or $\ast\times\ast;$
  \item[(2)] $^{\partial}w_{D}^{\bullet}(z)$ has no gaps of size 2 or more for any $\bullet\in \mathcal{B}.$
\end{description}

\begin{proof}
For any element $e$ and subset $A$ of $E$, we observe that $w(D^{\bullet|A\Delta e})$ and $w(D^{\bullet|A})$ differ by at most one for $\bullet\in \{\times, \ast\times\ast\}$, and by at most two for $\bullet\in \{\ast, \ast\times, \times\ast\}$ by Theorem \ref{lem 5}. This yields statements (1) and (2).
\end{proof}
\end{proposition}

\begin{remark}
There exists a vf-safe delta-matroid $D$ such that $^{\partial}w_{D}^{\bullet}(z)$ is not interpolating for $\bullet\in\{\ast, \ast\times, \times\ast\}$. For example, let $$D_{1}=(\{1, 2\}, \{\emptyset, \{1, 2\}\})$$ and $$D_{2}=(\{1, 2\}, \{\emptyset, \{1\}, \{1, 2\}\}).$$ We have $$^{\partial}w_{D_{1}}^{\ast}(z)=2+2z^{2}$$ and $$^{\partial}w_{D_{2}}^{\ast\times}(z)={^{\partial}w_{D_{2}}^{\times\ast}(z)}=1+3z^{2}.$$
\end{remark}



\section{ Partial-twuality polynomials and intersection graphs}

In \cite{QYXJ2}, we showed that two bouquets with the same intersection graph have the same partial-$\delta$ polynomial. In this section, we prove that the intersection graphs can determine the partial-twuality polynomials of bouquets and normal binary delta-matroids, respectively. Let $\eta: \mathcal{R}\rightarrow \mathcal{B}$ be the group isomorphism induced by $\eta(\delta)=\ast$, and $\eta(\tau)=\times$.

\begin{lemma}[\cite{CMNR}\label{lem 1}]
If $G$ is a ribbon graph, $A\subseteq E$ and $\bullet\in \mathcal{R}$. Then $$D(G^{\bullet|A})=D(G)^{\eta(\bullet)|A}$$
and $$\varepsilon(G)=w(D(G)).$$
\end{lemma}

\begin{proposition}
Let $G=(V, E)$ be a ribbon graph and $\bullet\in \mathcal{R}$. Then $$^{\partial}w_{D(G)}^{\eta(\bullet)}(z)={^{\partial}\varepsilon_{G}^{\bullet}(z)}.$$
\end{proposition}

\begin{proof}
By Lemma \ref{lem 1}, for any $A\subseteq E$, $$w(D(G)^{\eta(\bullet)| A})=w(D(G^{\bullet|A}))=\varepsilon(G^{\bullet|A}).$$
Hence $^{\partial}w_{D(G)}^{\eta(\bullet)}(z)={^{\partial}\varepsilon_{G}^{\bullet}(z)}$.
\end{proof}



\begin{theorem}\label{main-2}
If two normal binary delta-matroids $D$ and $\widetilde{D}$ have the same intersection graph, then $^{\partial}w_{D}^{\bullet}(z)={^{\partial}w_{\widetilde{D}}^{\bullet}(z)}$ for any  $\bullet\in \mathcal{B}$.
\end{theorem}

\begin{proof}
Since $G_{D}=G_{\widetilde{D}}$, $D=D(A_{G_{D}})$ and $\widetilde{D}=D(A_{G_{\widetilde{D}}})$, we have $D=\widetilde{D}$. Thus $^{\partial}w_{D}^{\bullet}(z)={^{\partial}w_{\widetilde{D}}^{\bullet}(z)}$ for any  $\bullet\in \mathcal{B}$.
\end{proof}

\begin{theorem}\label{main-1}

Let $B$ and $\widetilde{B}$ be two bouquets. If $G_{D(B)}=G_{D(\widetilde{B})}$, then $^{\partial}\varepsilon_{B}^{\bullet}(z)={^{\partial}}\varepsilon_{\widetilde{B}}^{\bullet}(z)$ for any $\bullet\in \mathcal{R}$.
\end{theorem}

\begin{proof}
Since $G_{D(B)}=G_{D(\widetilde{B})}$, it follows that $D(B)=D(\widetilde{B})$. For any $A\subseteq E(B)$, we denote its corresponding
subset of $E(\widetilde{B})$ by $\widetilde{A}$, then
$$D(B^{\bullet|A})=D(B)^{\eta(\bullet)|A}=D(\widetilde{B})^{\eta(\bullet)|\widetilde{A}}=D(\widetilde{B}^{\bullet|\widetilde{A}}),$$
by Lemma \ref{lem 1}. We have $$w(D(B^{\bullet|A}))=w(D(\widetilde{B}^{\bullet|\widetilde{A}})).$$
Since $w(D(B^{\bullet|A}))=\varepsilon(B^{\bullet|A})$ and $w(D(\widetilde{B}^{\bullet|\widetilde{A}}))=\varepsilon(\widetilde{B}^{\bullet|\widetilde{A}})$, it follows that $\varepsilon(B^{\bullet|A})=\varepsilon(\widetilde{B}^{\bullet|\widetilde{A}}).$ Thus $^{\partial}\varepsilon_{B}^{\bullet}(z)={^{\partial}\varepsilon_{\widetilde{B}}^{\bullet}(z)}.$
\end{proof}

\section{Partial-twuality monomials}
We \cite{QYXJ3, QYXJ4} showed that a normal binary delta-matroid whose partial-$\ast$ polynomials have only one term if and only if each connected component of the intersection graph of the delta-matroid is either a complete graph of odd order or a single vertex with a loop. In this section, we give a characterization of vf-safe delta-matroids whose partial-$\times$ and $\ast\times\ast$ polynomials have only one term.


\begin{lemma}[\cite{BRHH}]\label{lem 8}
Let $D=(E, \mathcal{F})$ be a set system and $X, Y\subseteq E.$ We have $Y\in \mathcal{F}(D^{\times|X})$ if and only if $|\{Z\in \mathcal{F}(D)~|~Y-X\subseteq Z\subseteq Y\}|$ is odd.
\end{lemma}

\begin{theorem}
Let $D=(E, \mathcal{F})$ be a vf-safe delta-matroid. Then
\begin{description}
  \item[(1)] $^{\partial}w_{D}^{\times}(z)=cz^{m}$ if and only if $\mathcal{F}(D)=\{E\};$
  \item[(2)] $^{\partial}w_{D}^{\ast\times\ast}(z)=cz^{m}$ if and only if $\mathcal{F}(D)=\{\emptyset\}$.
\end{description}
\end{theorem}

\begin{proof}
{\bf (1)} Suppose that  $^{\partial}w_{D}^{\times}(z)=cz^{m}$. Then for any $e\in E$, the dual type of $e$ is $u$ in $D$, otherwise applying $\times|e$ changes the width according to Theorem \ref{lem 5}. Then for any $A\in \mathcal{F}_{max}(D)\cup \mathcal{F}_{max-1}(D)$, we have $e\in A$. Thus $\mathcal{F}_{max}(D)=\{E\}$ and $\mathcal{F}_{max-1}(D)=\emptyset$. Suppose $\mathcal{F}_{max-2}(D)\neq\emptyset$. Let $B\in \mathcal{F}_{max-2}(D)$ and $f\in E-B$. Then $B\cup f, E\in\mathcal{F}(D^{\times|f})$ by Lemma \ref{lem 8}. Observe that $B\cup f\in\mathcal{F}_{max-1}(D^{\times|f})$ and $E\in \mathcal{F}_{max}(D^{\times|f})$.
Let $g\in E-(B\cup f)$. Then there exists $B\cup f\in \mathcal{F}_{max-1}(D^{\times|f})\cup \mathcal{F}_{max}(D^{\times|f})$ such that $g\notin B\cup f$. Thus the dual type of $g$ is not $u$ in $D^{\times|f}$. We have $w(D^{\times|f})\neq w((D^{\times|f})^{\times|g})$ by Theorem \ref{lem 5}. Then $^{\partial}w_{D^{\times|f}}^{\times}(z)\neq cz^{m}$. Note that $^{\partial}w_{D^{\times|f}}^{\times}(z)={^{\partial}w_{D}^{\times}(z)}$ by Proposition \ref{pro 1}. It follows that $^{\partial}w_{D}^{\times}(z)\neq cz^{m}$, a contradiction. Then $\mathcal{F}_{max-2}(D)=\emptyset$.   Since the maximum gap in the collection of sizes of feasible sets of a delta-matroid is two, it follows that $\mathcal{F}(D)=\{E\}$.

Conversely, for any $X\subseteq E$, $$\mathcal{F}_{min}(D^{\times|X})=\mathcal{F}_{min}(D)=\{E\}$$ by Lemma \ref{lem 2}. Then $\mathcal{F}(D^{\times|X})=\{E\}$. Thus $w(D^{\times|X})=0$ and $^{\partial}w_{D}^{\times}(z)=2^{|E|}.$

{\bf (2)} For $\ast\times\ast$, by Theorem \ref{th 2}, $^{\partial}w_{D}^{\ast\times\ast}(z)={^{\partial}w_{D^{*}}^{\times}(z)=cz^{m}}$ if and only if
$\mathcal{F}(D^{\ast})=\{E\}$ if and only if $\mathcal{F}(D)=\{\emptyset\}$.
\end{proof}

\section*{Acknowledgements}
This work is supported by NSFC (Nos. 12171402, 12101600) and the Fundamental Research Funds for the Central Universities (No. 2021QN1037).


\end{document}